\documentclass[a4paper,10pt]{scrartcl}

\usepackage[utf8x]{inputenc}
\usepackage[T1]{fontenc}
\usepackage[ngerman,english]{babel}
\usepackage{amsmath, amsthm, amssymb} 
\usepackage{mathtools}
\usepackage[usenames,dvipsnames]{color}
\usepackage{tikz}
\usepackage{fullpage}
\allowdisplaybreaks

\newcommand{\nn}{\nonumber}

\newcommand{\eps}{\varepsilon}

\newcommand{\R}{\mathbb{R}}
\newcommand{\Rn}{\R^n}

\newcommand{\ny}{\nu}

\newcommand{\phii}{\varphi}

\newcommand{\al}{\alpha}

\newcommand{\col}{\colon}
\newcommand{\Om}{\Omega}
\newcommand{\na}{\nabla}

\newcommand{\Laplace}{\Delta}
\newcommand{\Lap}{\Laplace}

\newcommand{\calF}{\mathcal{F}}

\newcommand{\rand}{\del\Omega}
\newcommand{\Ombar}{\overline{\Om}}
\newcommand{\dom}{\del \Om}
\newcommand{\amrand}{|_{\rand}}
\newcommand{\bdry}{\amrand}
\DeclareMathOperator{\diam}{diam}

\newcommand{\del}{\partial}

\newcommand{\delny}{\partial_\ny}
\newcommand{\intom}{\int_\Om}
\newcommand{\io}{\intom}
\newcommand{\intdom}{\int_{\dom}}

\newcommand{\intnt}{\int_0^t}
\newcommand{\intntau}{\int_0^\tau}

\newcommand{\Tmax}{T_{max}}

\newcommand{\Liom}{L^\infty(\Om)}
\newcommand{\Lqom}{L^q(\Om)}
\newcommand{\Lpom}{L^p(\Om)}

\newcommand{\Lzom}{L^2(\Om)}
\newcommand{\Leom}{L^1(\Om)}
\newcommand{\Lom}[1]{L^{#1}(\Om)}

\newcommand{\Weqom}{W^{1,q}(\Om)}

\newcommand{\ddt}{\frac{\rm d}{{\rm d}t}}

\newcommand{\norm}[2][]{\left\|#2\right\|_{#1}}

\newcommand{\upto}{\nearrow}

\newcommand{\set}[1]{\{#1\}}
\newcommand{\bigset}[1]{\Big\{#1\Big\}}

\newcommand{\sub}{\subset}

\newtheorem{theorem}{Theorem}
\numberwithin{theorem}{section}
\newtheorem{lemma}[theorem]{Lemma}

\newtheorem{remark}[theorem]{Remark}

\newcommand{\Fab}{\calF_{a,b}}

\title{A new approach toward boundedness in a two-dimensional parabolic chemotaxis system with singular sensitivity}
\author{Johannes Lankeit\thanks{Institut f\"ur Mathematik, Universit\"at Paderborn, Warburger Str. 100, 33098 Paderborn, Germany; email: \mbox{johannes.lankeit@math.upb.de}}}

\begin{document}
 \maketitle
 \begin{abstract}
  We consider the parabolic chemotaxis model
\[
 \begin{cases}
  u_t=\Lap u - \chi \na\cdot(\frac uv \na v)\\
  v_t=\Lap v - v + u
 \end{cases}
\]
 in a smooth, bounded, convex two-dimensional domain and show global existence and boundedness of solutions for $\chi\in(0,\chi_0)$ for some $\chi_0>1$, thereby proving that the value $\chi=1$ is not critical in this regard.\\
Our main tool is consideration of the energy functional 
\[ \Fab(u,v)=\io u\ln u - a \io u\ln v + b \io |\na \sqrt{v}|^2 \]
for $a>0$, $b\geq 0$, where using nonzero values of $b$ appears to be new in this context.
 \end{abstract}

\section{Introduction}

Numerous phenomena in connection with spontaneous aggregation can be described by PDE models incorporating a cross-diffusion mechanism. A prototypical example, which lies at the core of models used for a variety of purposes and to so different aims as the description pattern formation of bacteria or slime mold in biology \cite{KS} or the prediction of burglary in criminology \cite{MPS}, is the following variant of the Keller-Segel system of chemotaxis:
\begin{align}
 \label{eq:CT}
 u_t=&\Lap u- \na\cdot (u\, S(v)\, \na v)\nn\\
 v_t=&\Lap v-v+u\\
 \delny u\bdry=&\delny v\bdry=0\nn\\
 u(\cdot,0)=u_0, &\; v(\cdot,0)=v_0\nn
\end{align}
in a bounded domain $\Om\sub \R^n$ with smooth boundary, with given nonnegative initial data $u_0,v_0$. 
We shall be concerned with the case of the singular sensitivity function $S$ given by 
\begin{equation}
 \label{eq:S}
 S(v)=\frac\chi v
\end{equation}
for a constant $\chi>0$, which is in compliance with the Weber-Fechner law of stimulus perception (see \cite{KStrav}).

One of the first questions of mathematical interest with respect to this model is that of existence of a global classical solution, as opposed to blow-up of solutions in finite time.
For the vast mathematical literature on chemotaxis, a large part of which is concerned with this question, see one of the survey articles \cite{Horstmann03,Horstmann04,HillenPainter,BBTW} and references therein.

According to the standard reasoning in the realm of chemotaxis equations (as e.g. formulated in \cite{BBTW}), in order to obtain global existence of classical solutions, for the two-dimensional case considered here, it is sufficient to derive $t$-independent bounds on the quantities $\intom u(t)\ln u(t)$ and $\int |\na v(t)|^2$.

To achieve this in the particular context of \eqref{eq:CT}, it has proven useful to consider the expression 
\begin{equation}\label{eq:fctshort}
 \intom u\ln u - a \intom u\ln v,
\end{equation}
as it has been done by Nagai, Senba, Yoshida \cite{NSY1997} or Biler \cite{Biler99}. In these works, global existence of solutions has been derived for $\chi\leq1$. 

In the present article we shall answer the question whether $\chi=1$ is a critical value in this regard in the negative. 
This question had been left open in \cite{Wk}, where the above-mentioned results have been generalized to higher dimension $n$, then obtaining existence in the case $\chi<\sqrt{2/n}$. 

Let us mention some more results concerning equation \eqref{eq:CT}: 
That the classical solutions for $\chi<\sqrt{2/n}$ are global-in-time bounded has been shown in \cite{Fujie}.
In \cite{Wk} also weak solutions have been shown to exist for \eqref{eq:CT}, as long as $\chi<\sqrt{\frac{n+2}{3n-4}}$.
In the radially symmetric setting, moreover, certain global weak ``power-$\lambda$-solutions'' exist (\cite{WkCSt}).\\ 
Related parabolic-elliptic chemotaxis models are investigated, e.g. in \cite{fujie_winkler_yokota_logsource}, where the presence of terms describing logistic growth is used to ensure global existence and boundedness of classical solutions. In \cite{fujie_winkler_yokota} global existence and boundedness of classical solutions to the parabolic-elliptic counterpart of \eqref{eq:CT} are obtained for even more singular sensitivities of the form $0<S(v)\leq\frac\chi{v^k}$, $k\geq 1$, under a smallness condition on $\chi$, which for $k=1$ and $n=2$  amounts to $\chi<1$.\\
Also concerning classical solutions of the fully parabolic system \eqref{eq:CT}, to the best of our knowledge, the assertions for $\chi\leq1$ are the best known so far.

Since the new possible values for $\chi$ are but slightly larger than $1$, rather than these values it is the method that can be considered the new contribution of the present article: Key to our approach toward the expansion of the interval of values for $\chi$ 
known to yield global solutions, namely,  
shall be the employment of an additional summand
\[
 b\int|\na \sqrt{v}|^2
\]
in \eqref{eq:fctshort}. Functionals containing this term have successfully been used in the context of coupled chemotaxis-fluid systems (see \cite{wkfluid})
or of chemotaxis models with consumption of the chemoattractant \cite{TaoWk}
(e.g. obtained from the aforementioned system upon neglection of the fluid).

In the end we will arrive at the following
\begin{theorem}
 \label{thm:main}
 Let $\Om\sub\R^2$ be a convex bounded domain with smooth boundary. Let $0\leq u_0\in C^0(\Ombar)$, $u_0\not\equiv 0$,  
 $0<v_0\in \bigcup_{q>2} W^{1,q}(\Om)$.
 Then there exists $\chi_0>1$ such that for any $\chi\in(0,\chi_0)$ the system \eqref{eq:CT} has a global classical solution, which is bounded. 
\end{theorem}

The plan of the paper is as follows: In the next section we will discuss local existence of and an extensibility criterion for solutions to \eqref{eq:CT}. 
Section \ref{sec:estimates} provides identities and estimates that will facilitate the usage of the additional term at the center of the proof of Theorem \ref{thm:main}, to which Section \ref{sec:proofmainthm} will be devoted.

\section{How to ensure global existence}

A general existence theorem for chemotaxis models is the following, taken from \cite{BBTW}:
\begin{theorem}\label{thm:exBBTW}
 Let $n\geq 1$ and $\Om\sub\Rn$ be a bounded domain with smooth boundary and let $q>n$. For some $\omega\in(0,1)$ let $S\in C^{1+\omega}_{loc}(\Ombar\times[0,\infty)\times\R^2)$, $f\in C^{1-}(\Ombar\times[0,\infty)\times\R^2)$ and $g\in C^{1-}_{loc}(\Ombar\times[0,\infty)\times\R^2)$, and assume that $f(x,t,0,v)\geq 0$ for all $(x,t,v)\in\Ombar\times[0,\infty)^2$ and that $g(x,t,u,0)\geq 0$ for any $(x,t,u)\in\Ombar\times[0,\infty)^2$.
 Then for all nonnegative $u_0\in C^0(\Ombar)$ and $v_0\in W^{1,q}(\Om)$ there exist $\Tmax\in(0,\infty]$ and a uniquely determined pair of nonnegative functions
\begin{align}\label{eq:solnspaces}
 u\in& C^0(\Ombar\times[0,\Tmax))\cap C^{2,1}(\Ombar\times(0,\Tmax)),\\
 v\in & C^0(\Ombar\times[0,\Tmax))\cap C^{2,1}(\Ombar\times(0,\Tmax))\cap L_{loc}^\infty([0,\Tmax); \Weqom),\nn
\end{align}
such that $(u,v)$ solves 
\begin{align}\label{eq:CTallg}
 u_t=&\Lap u-\na\cdot\left(uS(x,t,u,v)\na v\right) + f(x,t,u,v),\\
v_t=&\Lap v-v+g(x,t,u,v),\nn\\
0=&\delny u\bdry=\delny v\bdry,\nn\\
&u(\cdot,0)=u_0, v(\cdot,0)=v_0\nn
\end{align}
classically in $\Om\times(0,\Tmax)$ and such that 
\begin{equation}\label{eq:atTmax}
 \mbox{if } \Tmax<\infty, \mbox{ then } \norm[\Liom]{u(\cdot,t)}+\norm[\Weqom]{v(\cdot,t)}\to \infty \;\mbox{as }t\upto\Tmax.
\end{equation}
\end{theorem}
\begin{proof}
 A Banach-type fixed point argument provides existence of mild solutions on a short time interval whose length $T$ depends on $\norm[\infty]{u_0}$, $\norm[W^{1,q}]{v_0}$. Standard bootstrapping arguments ensure the regularity properties listed above. It follows from the dependence of $T$ on the norms of $u_0$ and $v_0$ that the solution can be extended to $\Tmax\in(0,\infty]$ satisfying \eqref{eq:atTmax}, 
 see \cite[Lemma 4.1]{BBTW}.
\end{proof}
This theorem is not directly applicable to \eqref{eq:CT}, because it does not cover the case of singular functions $S$. 
We will remove this obstruction via use of the following lemma, which is a generalization of Lemma 2.2 of \cite{Fujie}.

\begin{lemma}\label{lem:vunifbdbelow}
Let the conditions of Theorem \ref{thm:exBBTW} be satisfied and let $\zeta>0$.\\ 
Then there is $\eta=\eta(u_0,v_0,\zeta)>0$ such that if $v_0$ and the solution $(u,v)$ to \eqref{eq:CTallg} satisfy
\[
 \inf v_0>0  \qquad \mbox{ and }\qquad \inf_{s\in[0,\Tmax)} \io g(x,s,u(x,s),v(x,s)) dx\geq \zeta,
\]
the second component of the solution also fulfils 
\[
 v(x,t)\geq \eta \qquad \mbox{for all } (x,t)\in\Ombar\times[0,\Tmax).
\]
\end{lemma}
\begin{proof}
 Let us fix $\tau=\tau(u_0,v_0)>0$ such that 
\[
 \inf_\Om v(\cdot,t)\geq \frac12 \inf_\Om v_0 \qquad\mbox{ for all }t\in[0,\tau].
\]
Employing the pointwise estimate 
\[
 (e^{t\Lap}w)(x)\geq \frac1{(4\pi t)^{\frac n2}} e^{-\frac{d^2}{4t}}\io w\qquad \mbox{for nonnegative } w\in C^0(\Ombar)
\]
for the Neumann heat semigroup $e^{t\Lap}$ with $d=\diam\Om$, as provided in \cite[Lemma 2.2]{Fujie} following 
\cite[Lemma 3.1]{HPW}, we can then conclude 
that 
\begin{align*}
 v(\cdot,t)=&e^{t(\Lap-1)} v_0 + \intnt e^{(t-s)(\Lap-1)} g(\cdot,s,u(\cdot,s),v(\cdot,s))ds\\
\geq&\intnt \frac{1}{(4\pi (t-s))^{\frac n2}} e^{-\frac{d^2}{4(t-s)}-(t-s)} \io g(\cdot,s,u(\cdot,s),v(\cdot,s)) ds\\
\geq&\intnt \frac{1}{(4\pi r)^{\frac n2}} e^{-(r+\frac{d^2}{4r})} dr \inf_{s\in[0,t]} \io g(x,s,u(x,s),v(x,s)) dx\\
\geq& \zeta\intntau \frac{1}{(4\pi r)^{\frac n2}} e^{-(r+\frac{d^2}{4r})} dr \qquad \mbox{ in }\Om
\end{align*}
for any $t\in[\tau,\Tmax)$. 
With $\eta=\min\set{\frac{\inf_\Om v_0}2,\zeta\int_0^{\tau(u_0,v_0)} \frac{1}{(4\pi r)^{\frac n2}}e^{-(r+\frac{d^2}{4r})} dr}$ this proves the claim.
\end{proof}

With this lemma we can weaken the assumptions on the sensitivity $S$ so as to allow for a singularity at $v=0$.

\begin{theorem}\label{thm:exSingular}
i) Let $S\in C^{1+\omega}_{loc}(\Ombar\times[0,\infty)\times\R\times(0,\infty))$ for some $\omega\in(0,1)$ and apart from the condition on $S$ let the assumptions of Theorem \ref{thm:exBBTW} be satisfied.\\
Additionally, assume that $f$ is nonnegative and $g(x,t,u,v)\geq cu$ for some $c>0$ and any $(x,t,u,v)\in \Ombar\times[0,\infty)\times\R^2$ 
and that $\inf_\Om v_0>0$ and $\io u_0=:m>0$. 
Then there is $\Tmax>0$ such that \eqref{eq:CT} has a unique solution $(u,v)$ as in \eqref{eq:solnspaces} and such that \eqref{eq:atTmax} holds.\\
ii) Furthermore, 
if there are $K_1,K_2>0$ such that $f(x,t,u,v)\leq K_1$ and $g(x,t,u,v)\leq K_2(1+u)$ for all $(x,t,u,v)\in \Om\times(0,\infty)^3$, and for every $\eta>0$, $|S|$ is bounded on $\Om\times(0,\infty)^2\times (\eta,\infty)$, and 
if $n=2$ and there is $M>0$ such that 
\begin{equation}\label{eq:intsbd}
 \io u(\cdot,t)\ln u(\cdot,t)\leq M,\qquad\mbox{and}\qquad \io |\na v(\cdot,t)|^2 \leq M\qquad\mbox{ for all }t\in[0,\Tmax)
\end{equation}
then $(u,v)$ is global and bounded.
\end{theorem}

\begin{proof}
i) Let $\eta:=\eta(u_0,v_0,cm)$ be as in Lemma \ref{lem:vunifbdbelow}. Let $\zeta\col \R\to [0,1]$ be a smooth, monotone decreasing function with $\zeta(\frac{\eta}2)=1$ and $\zeta(\eta)=0$. Define 
\[S_\eta(x,t,u,v):=\begin{cases}S(x,t,u,\frac\eta2),& (x,t,u,v)\in \Ombar\times[0,\infty)\times\R\times(-\infty,\frac\eta2),\\
        \zeta(v)S(x,t,u,\frac\eta2)+(1-\zeta(v))S(x,t,u,v),& (x,t,u,v)\in \Ombar\times[0,\infty)\times\R\times[\frac\eta2,\infty).\end{cases}
       \]
Then $S_\eta\in C^{1+\omega}_{loc}(\Ombar\times[0,\infty)\times\R^2)$ and $S$ and $S_\eta$ agree for $v\geq \eta$. 
Let us denote by \eqref{eq:CTallg}$_\eta$ problem \eqref{eq:CTallg} with $S$ replaced by $S_\eta$. Then we can apply Theorem \ref{thm:exBBTW} to \eqref{eq:CTallg}$_\eta$ and obtain a solution $(u,v)$ with the required properties \eqref{eq:solnspaces} and \eqref{eq:atTmax}. 
Nonnegativity of $f$ and integration of the first equation of \eqref{eq:CTallg}$_\eta$ entail that $\io u(t)\geq m$ for all $t\in[0,\Tmax)$ and accordingly $\io g(x,t,u(x,t),v(x,t)) dx\geq cm>0$ for all $t\in [0,\Tmax)$. 
Therefore, by Lemma \ref{lem:vunifbdbelow}, $v\geq \eta$ and hence $(u,v)$ solves \eqref{eq:CTallg} as well.\\
In order to carry over the uniqueness statement from Theorem \ref{eq:CTallg}, we ensure that any solution of \eqref{eq:CTallg} also solves \eqref{eq:CTallg}$_\eta$ in $\Om\times[0,\Tmax)$: 
Let $v$ be a solution of \eqref{eq:CTallg}. Let $\eps\in(0,\frac\eta2)$ and define $t_0=\inf\set{t: \inf_\Om v(t)<\eps}\in(0,\infty]$. Then $(u,v)$ solves \eqref{eq:CTallg}$_\eta$ in $(0,t_0)$. Assume $t_0<\infty$. Then by Lemma \ref{lem:vunifbdbelow} and continuity of $v$, $v(x,t_0)\geq \eta>\eps=\inf_\Om v(\cdot,t)$ for all $x\in \Om$, a contradiction. \\
ii) Since $(u,v)$ is a solution of \eqref{eq:CTallg}$_\eta$, we can apply \cite[Lemma 4.3]{BBTW}, which turns \eqref{eq:intsbd} into a uniform-in-time bound on $\norm[\Liom]{u(\cdot,t)}+\norm[\Weqom]{v(\cdot,t)}$, thus asserting global existence by means of \eqref{eq:atTmax} and boundedness.
\end{proof}

\begin{remark}
 Throughout the remaining part of the article, we will assume that $\Om\sub \R^2$ is a bounded, smooth domain, that $0\leq u_0\in C^0(\Ombar)$, $q>2$ and $v_0\in W^{1,q}(\Om)$, $\inf_\Om v_0>0$ and $\io u_0=:m>0$.\\
 Then, in particular, Theorem \ref{thm:exSingular} is applicable to \eqref{eq:CT}. 
 Furthermore, any solution $(u,v)$ of \eqref{eq:CT} satisfies 
\begin{equation}\label{eq:intuconst}
 \io u(t) = m \qquad \mbox{for all }t\in[0,\Tmax).
\end{equation}
\end{remark}

For the purpose of using it in the next proof, let us recall the well-known Gagliardo-Nirenberg inequality: 
\begin{lemma}\label{lem:GN}
 Let $\Om\sub \Rn$ be a bounded smooth domain. Let $j\geq 0, k\geq 0$ be integers and $p,q,r,s>1$. 
 There are constants $c_1, c_2>0$ such that for any function $w\in \Lqom \cap L^s(\Om)$ with $D^k w\in L^r(\Om)$, 
\[
 \norm[p]{D^j w}\leq c_1\norm[r]{D^k w}^{\al} \norm[q]{w}^{1-\al} + c_2\norm[s]{w}, 
\]
whenever 
\(
 \frac1p=\frac jn+\left(\frac1r-\frac kn\right) \al + \frac{1-\al}q\) and \(\frac jk\leq \al< 1.
\)
\end{lemma}
\begin{proof}
 Cf. \cite[p. 126]{Nirenberg_59}
\end{proof}

\begin{lemma}\label{lem:intandintgivesGE}
 Let $(u,v)$ be a solution to \eqref{eq:CT}, let $\tau=\min\set{1,\frac{\Tmax}2}$, and assume there exists $C>0$ such that 
\[
 \int_t^{t+\tau} \io \frac{|\na u|^2}u \leq C\quad \mbox{for any }t\in (0,\Tmax-\tau) 
\]
and that 
\[
 \io u(t)\ln u(t) \leq C \quad \mbox{for any }t\in(0,\Tmax)
\]
Then $\Tmax=\infty$ and $(u,v)$ is bounded.
\end{lemma}
\begin{proof}
Let $c_1,c_2>0$ be the constants yielded by Lemma \ref{lem:GN} for $j=0$, $k=1$, $q=2$, $r=2$, $\al=\frac12$. Then with $m$ from \eqref{eq:intuconst},  
\begin{align*}
 \int_t^{t+\tau}\io u^2 =& \int_t^{t+\tau}\norm[\Lom4]{\sqrt{u}}^4\leq\int_t^{t+\tau}\left(c_1\norm[\Lzom]{\na \sqrt{u}}^\al\norm[\Lzom]{\sqrt{u}}^{1-\al} + c_2\norm[\Lzom]{\sqrt u}\right)^4\\\leq& \int_t^{t+\tau} \left(c_1  \left(\io \frac{|\na u|^2}{4u} \right)^{\frac14}m^{\frac14} + c_2m^{\frac12}\right)^4\leq \frac{c_1^4mC}4+c_2^4m^2
\end{align*}
holds for any $t\in(0,\Tmax-\tau)$.\\
Multiplying the second equation of \eqref{eq:CT} by $-\Lap v$ and integrating, from Young's inequality we obtain 
\[
 \io |\na v(t)|^2 \leq \io |\na v_0|^2 - \intnt\io |\Lap v|^2 - \intnt\io |\na v|^2 +\frac12\intnt\io u^2 + \frac12 \intnt\io |\Lap v|^2 \qquad\mbox{on }(0,\Tmax), 
\]
that is, $y(t):=\io |\na v(t)|^2$ satisfies the differential inequality $y'+y\leq f$ on $[0,\Tmax)$, where $f=\frac12\io u^2$ satisfies $\int_t^{t+\tau} |f(s)| ds\leq C$ for all $t\in (0,\Tmax-\tau)$ and with some constant $C>0$. Let $z$ be a solution to $z'+z=f$, $z(0)=z_0=\io |\na v_0|^2$ and observe that the variation-of-constants formula entails 
\begin{align*}
 z(t)-e^{-t} z_0 =& \intnt e^{-s}f(t-s) ds \leq \sum_{k=0}^{\lfloor t/\tau\rfloor-1} \int_{k\tau}^{(k+1)\tau} e^{-s} |f(t-s)| ds + \int_{\tau\lfloor t/\tau\rfloor}^t e^{-s} |f(t-s)| ds\\
 \leq& \sum_{k=0}^{\lfloor t/\tau\rfloor-1} e^{-k\tau} C  + C \leq C \left( 1+ \frac 1{1-e^{-\tau}}\right) \qquad \mbox{for }t\in(0,\Tmax),
\end{align*}
so that an ODE comparison yields boundedness of  $y=\io |\na v(t)|^2$.\\
Together with the second assumption, the bound on $\io u\ln u$, this is sufficient to conclude global existence and boundedness of solutions by Theorem \ref{thm:exSingular} ii).
\end{proof}

\section{Some useful general estimates and identities}
\label{sec:estimates}

\begin{lemma}
 Let $\Om$ be convex and let $w\in C^2(\Ombar)$ satisfy $\delny w\bdry=0$. Then for all $x\in \dom$ also $\delny |\na w(x)|^2\leq 0$.
\end{lemma}
\begin{proof}
 This is Lemme 2.I.1 of \cite{Lions}.
\end{proof}

 \begin{lemma}
 For all positive $w\in C^2(\Ombar)$ satisfying $\delny w\bdry=0$ 
 \[
  \io w |D^2 \ln w|^2=\io \frac1w |D^2 w|^2  +\io \frac1{w^2} |\na w|^2 \Lap w - \io \frac{|\na w|^4}{w^3}
 \]
 \end{lemma}
 \begin{proof}
  This proof is also contained in the proof of \cite[Lemma 3.2]{wkfluid}. The equality rests on the pointwise identity
\[
 w|D^2 \ln w|^2 = w \left( -\frac1{w^2} |\na w|^2 + \frac 1w D^2 w \right)^2=\frac{|\na w|^4}{w^3} +\frac1w|D^2 w|^2 - \frac1{w^2} \na|\na w|^2\cdot \na w
\]
and integration by parts in the last term giving
\[
 - \io \frac1{w^2} \na|\na w|^2\cdot \na w = \io \frac1{w^2}|\na w|^2\Lap w - 2 \io \frac{|\na w|^4}{w^3}.
\]
\end{proof}

\begin{lemma}
 Let $w\in C^2(\Ombar)$ be positive and satisfy $\delny w\bdry=0$. Then
 \[
  -\io \frac1w |\Lap w|^2 = -\io \frac1w |D^2 w|^2 - \frac32 \io \frac{|\na w|^2\Lap w}{w^2} + \frac12 \io \frac{|\na w|^4}{w^3} + \frac12 \intdom \frac1w \delny |\na w|^2.
 \]
\end{lemma}
\begin{proof}
 This results from \cite[Lemma 3.1]{wkfluid} upon the choice of $h(w)=\frac1w$. The proof can be found in \cite[Lemma 2.3]{dPGG}.
\end{proof}

\begin{lemma}\label{lem:iofraclapvv}
(i) For all positive $w\in C^2(\Ombar)$ satisfying $\delny w\bdry=0$, 
 \[
  -\io \frac1w |\Lap w|^2 = -\io w|D^2 \ln w|^2-\frac12 \io \frac1{w^2}|\na w|^2 \Lap w +\frac12 \intdom \frac1w\delny |\na w|^2.
 \]
(ii) If furthermore $\Om$ is convex, then
\[
 -\io \frac1w |\Lap w|^2 \leq -\io w|D^2 \ln w|^2-\frac12 \io \frac1{w^2}|\na w|^2 \Lap w. 
\]
\end{lemma}
\begin{proof}
This is a direct consequence of the previous three lemmata.
\end{proof}

\begin{lemma}\label{lem:estvDDlnv}
 There is $c_0>0$ such that for all positive $w\in C^2(\Ombar)$ fulfilling $\delny w\bdry=0$ the following estimate holds:
 \[
  \io w |D^2 \ln w|^2 \geq c_0 \io \frac{|\na w|^4}{w^3}.
 \]
\end{lemma}
\begin{proof}
 An even more general version of this lemma and its proof can be found in \cite[Lemma 3.3]{wkfluid}.
\end{proof}
\begin{remark}\label{rem:c0}
 As can be seen from the referenced lemma, the constant in the above statement can be chosen to be $\frac1{(2+\sqrt{2})^2}$.
\end{remark}

\section{The energy functional. Proof of Theorem \ref{thm:main}}
\label{sec:proofmainthm}
In this section let us investigate the energy functional defined by 
\begin{equation}\label{eq:defF}
 \Fab(u(t),v(t))=\io u(t)\ln u(t) - a \io u(t)\ln v(t) + b \io |\na \sqrt{v(t)}|^2, \qquad t\in[0,\Tmax), 
\end{equation}
for nonnegative parameters $a,b$.

If we want to gain useful information from this functional, the upper bounds on its derivative that we will derive, should be accompanied by bounds for $\Fab$ from below.
In order to ensure those, let us first provide the following estimate for solutions of \eqref{eq:CT}.

\begin{lemma}\label{lem:intvp} \label{lem:intuconst}
Let $(u,v)$ be a solution to \eqref{eq:CT}. For any $p>0$ there is $C_p>0$ such that 
 \[
  \io v^p(t) \leq C_p \qquad \mbox{for any }t\in[0,\Tmax).
 \]
\end{lemma}
\begin{proof}
Since $t\mapsto \norm[\Leom]{u(t)}$ is constant by \eqref{eq:intuconst}, for $p\ge 1$ this is a consequence of Duhamel's formula for the solution of the second equation of \eqref{eq:CT} and estimates for the Neumann heat semigroup, which can e.g. be found in \cite[Lemma 1.3]{wk_aggrvsglobdiffbeh}: They provide $C>0$ such that for all $t\in(0,\Tmax)$, 
 \begin{align*}
  \norm[L^p(\Om)]{v(t)} \leq& \norm[\Lpom]{e^{t(\Lap-1)} v_0} + \intnt \norm[\Lpom]{e^{(t-s)(\Lap-1)}(u(s)-m)} +\norm[\Lpom]{e^{-(t-s)} m}ds\\
  \leq& \norm[\Lpom]{v_0} + \intnt \left(C(1+(t-s)^{-\frac n2(1-\frac1p)}) e^{-(t-s)} \norm[\Leom]{u(s)-m} + e^{-(t-s)} m |\Om|^{\frac1p}\right)ds. 
 \end{align*}
 The case $p\in(0,1)$ then follows from $v^p\leq 1+v$.
\end{proof}

The following lemma gives bounds from below as well as means to turn boundedness of $\Fab(u,v)$ into boundedness of $\io u\ln u$.

\begin{lemma}\label{lem:Fabgeq}
 Let $a,b\geq 0$. For any solution $(u,v)$ of \eqref{eq:CT}, there is $\gamma\in\R$ such that 
\[
 \Fab(u,v)\geq \frac12 \io u\ln u - \gamma \qquad \mbox{on }(0,\Tmax). 
\]
\end{lemma}
\begin{proof}
 Denoting $m=\io u(t)$ as in \eqref{eq:intuconst}, we have 
\[
 \Fab(u,v) \geq \frac12\io u\ln u +\io u\ln\frac{u^{\frac12}}{v^a} = \frac12\io u\ln u+ m\io \left(-\ln\frac{v^a}{u^{\frac12}}\right) \frac u m,
\]
similar as in the proof of \cite[Thm. 3]{Biler99}. Hence, following an idea from the proof of \cite[Lemma 3.3]{NSY1997} in applying Jensen's inequality with the probability measure $\frac u m d\lambda$ and the convex function $-\ln$, we obtain 
\begin{align*}
 \Fab(u,v)\geq& \frac12\io u\ln u - m\ln \io \frac{v^au^{\frac12}}m \\
\geq& \frac12\io u\ln u - m\ln \left(\frac1m \left(\io v^{2a} \io u\right)^{\frac12}\right)\\
\geq& \frac12\io u\ln u +\frac m2\ln m - \frac m2\ln C_{2a}
\end{align*}
after applying H\"older's inequality and with $C_{2a}$ as in \ref{lem:intvp}.
\end{proof}

\begin{lemma}\label{lem:bdbelow}
Let $a,b\geq 0$. For any solution $(u,v)$ of \eqref{eq:CT},\\
i)  $\Fab(u,v)$ is bounded below.\\
ii) If $\sup_{t\in[0,\Tmax)} \Fab(u(t),v(t))<\infty$ then $\sup_{t\in[0,\Tmax)} \io u(t)\ln u(t)<\infty$.
\end{lemma}
\begin{proof}
 Both statements are immediate consequences of Lemma \ref{lem:Fabgeq}.
\end{proof}

Lemma \ref{lem:intvp} as well enables us to control the first two summands of $\Fab(u,v)$ from above by $\io \frac{u^2}v$.
\begin{lemma}\label{lem43}
 Let $(u,v)$ be a solution to \eqref{eq:CT} and let $a>0$. Then for any $\delta>0$ there is $c_\delta>0$ such that 
\[
 \io u\ln u - a\io u\ln v \leq \delta\io \frac{u^2}v+c_\delta \qquad \mbox{ on }[0,\Tmax).
\]
\end{lemma}
\begin{proof}
 Given $a>0$ let $\eps\in(0,1)$ be so small that $\frac{1+\eps-2a\eps}{1-\eps}>0$. 
 There is $C_\eps>0$ such that for any $x>0$ we have $\ln x\leq C_\eps x^\eps$.
 Therefore for any $\delta>0$ Young's inequality and Lemma \ref{lem:intvp} provide $C_\delta>0$ and $c_\delta>0$ satisfying 
 \begin{align*}
  \io u\ln u - a\io u\ln v  =& \io u\ln \frac u {v^a} \leq C_\eps \io \frac{u^{1+\eps}}{v^{a\eps}}  \leq  \delta \io (u^{1+\eps}v^{-\frac{1+\eps}2})^{\frac2{1+\eps}} + C_\delta \io  (v^{\frac{1+\eps-2 a\eps}2})^\frac{2}{1-\eps}\\
\leq&\delta \io \frac{u^2}v + c_\delta.\qedhere
 \end{align*}
\end{proof}

With these preparations, we turn to the time derivative of $\Fab(u,v)$, beginning with the already investigated first part:

\begin{lemma}\label{lem:ddtFa0}
For any $a\geq 0$ and any solution $(u,v)$ of \eqref{eq:CT}, 
\begin{align*}
 \ddt \calF_{a,0}(u,v)(t) =&  -\io \frac{|\na u|^2}u + (\chi+2a )\io\frac{\na u \cdot \na v} v -a(\chi+1)\io \frac{u |\na v|^2}{v^2}+ a \io u - a\io \frac{u^2}v
\end{align*}
holds on $(0,\Tmax)$.
\end{lemma}

\begin{proof}
Using the first equation of \eqref{eq:CT} in $\ddt \left(\io u\ln u - a\io u\ln v\right)$ and integrating by parts we obtain:
\begin{align*}
 \ddt \left(\io u\ln u - a\io u\ln v\right) = & \io u_t \ln u + \io u_t - a\io u_t \ln v - a \io \frac u v v_t\\
 =& - \io \frac{\na u}u \left(\na u -\chi \frac u v \na v\right) + a \io \frac {\na v}v \left(\na u - \chi \frac u v \na v\right) - a \io \frac u v \left(\Lap v- v +u\right)\\
=& -\io \frac{|\na u|^2}u + \chi \io\frac{\na u \cdot \na v} v + a \io \frac{\na u\cdot \na v}v - a\chi \io \frac{u |\na v|^2}{v^2}\\& + a \io \frac {\na u\cdot\na v}v- a\io \frac{u|\na v|^2}{v^2} + a\io u -a\io \frac{u^2}v.\qedhere
\end{align*}
\end{proof}
Since we do not know the sign of $\io \frac{\na u\cdot \na v}v$ and, in this situation, cannot control $\io \frac{u|\na v|^2}{v^2}$, we are left with  Young's inequality, hoping that the resulting coefficient $\frac{(\chi+2a)^2}4-a(\chi+1)$ of $\io \frac{u|\na v|^2}{v^2}$ turns out to be negative. 
This can be achieved if $\chi<1$.

However, it becomes possible to cope with larger parameters if $\io \frac{u|\na v|^2}{v^2}$ can be controlled, e.g. by having control over $\io \frac{|\na v|^4}{v^3}$ and $\io \frac{u^2}v$. 
The second term already being in place, fortunately, the first is one of the terms arising from the following:

\begin{lemma}\label{lem:ddtnasqrtv}
Let $\Om$ be convex. For any solution $(u,v)$ of \eqref{eq:CT},
\begin{align}
 4 \ddt \left(\io |\na \sqrt v|^2\right)\leq& -2 c_0 \io \frac{|\na v|^4}{v^3} - \io \frac{|\na v|^2} v +2\io \frac{\na u\cdot \na v}v- \io \frac{|\na v|^2 u}{v^2}
\end{align}
holds on $(0,\Tmax)$, where $c_0$ is the constant provided by Lemma \ref{lem:estvDDlnv}.
\end{lemma}
\begin{proof}
From the second equation of \eqref{eq:CT}, we obtain 
\begin{align*}
 4 \ddt \left(\io |\na \sqrt v|^2\right) =& \ddt \io \frac{|\na v|^2}v =\io \frac{2\na v\cdot \na v_t}v - \io \frac{|\na v|^2v_t} {v^2}\\
 =& \io \frac{2\na v\cdot \na \Lap v}v - \io \frac{2|\na v|^2} v +\io \frac{2\na v\cdot\na u}v - \io \frac{|\na v|^2 \Lap v}{v^2} +\io \frac{|\na v|^2 v}{v^2} - \io \frac{|\na v|^2 u}{v^2}.
\end{align*}
Integration by parts in the first integral and merging the second and second to last summand lead us to
\begin{align}\label{eq:ddtionasqrtv}
 4 \ddt \left(\io |\na \sqrt v|^2\right) 
= -2\io \frac{|\Lap v|^2}v + \io \frac{|\na v|^2\Lap v}{v^2} - \io \frac{|\na v|^2} v +2\io \frac{\na u\cdot \na v}v- \io \frac{|\na v|^2 u}{v^2}.
\end{align}
By Lemma \ref{lem:iofraclapvv} and due to the convexity of $\Om$ we can transform the first summand according to 
\[
 -2\io \frac{|\Lap v|^2}v \leq - 2\io v|D^2 \ln v|^2- \io \frac1{v^2}|\na v|^2 \Lap v,
\]
making the second term in the right hand side of \eqref{eq:ddtionasqrtv} vanish:
\begin{align}
 4 \ddt \left(\io |\na \sqrt v|^2\right) \leq -2\io v|D^2\ln v|^2  +2\io \frac{\na u\cdot \na v}v- \io \frac{|\na v|^2 u}{v^2}.
\end{align}
We are left with a term we can estimate with the help of Lemma \ref{lem:estvDDlnv}:
\[
 - 2\io v|D^2 \ln v|^2 \leq -2 c_0 \io \frac{|\na v|^4}{v^3},
\]
thereby gaining the term which will make the crucial difference in the estimates to come and arriving at 
\begin{align*}
 4 \ddt \left(\io |\na \sqrt v|^2\right)\leq& -2 c_0 \io \frac{|\na v|^4}{v^3} - \io \frac{|\na v|^2} v +2\io \frac{\na u\cdot \na v}v- \io \frac{|\na v|^2 u}{v^2}.\qedhere
\end{align*}
\end{proof}

If we combine the previous two lemmata, we are led to:
\begin{lemma}\label{lem42}
Let $\Om\sub\R^2$ be a convex, bounded, smooth domain and let $a,b\geq0$, $\delta\in(0,1)$. Then for any solution $(u,v)$ of \eqref{eq:CT},
\begin{align}\label{eq:lem47}\nn
 \ddt \Fab(u,v)(t) \leq&  \left( \frac1{4a(1-\delta)}\left(\frac{(\chi+2a+\frac b2)^2}{4 (1-\delta)}-a\chi-a -\frac b4\right)_+^2- \frac{b c_0}2 \right) \io \frac{|\na v|^4}{v^3} \\&-\delta \io \frac{|\na u|^2} u -\delta \io \frac{u^2}v + a \io u - \frac b4 \io \frac{|\na v|^2} v \qquad \mbox{ on }(0,\Tmax).
\end{align}
\end{lemma}

\begin{proof}
An estimate for $\ddt \Fab(u(t),v(t))$ is given by the sum of the terms from Lemma \ref{lem:ddtFa0} and Lemma \ref{lem:ddtnasqrtv}:
\begin{align*}
  \ddt \Fab(u,v)(t) \leq &  -\io \frac{|\na u|^2}u + \left(\chi+2a +\frac b 2\right)\io\frac{\na u \cdot \na v} v - \left(a\chi+a+\frac b4\right)\io \frac{u |\na v|^2}{v^2}\\&+ a \io u - a\io \frac{u^2}v
 -\frac b 2 c_0 \io \frac{|\na v|^4}{v^3} - \frac b4 \io \frac{|\na v|^2} v.
\end{align*}
In order to finally still have some control over $\int \frac{|\na u|^2}u$, as required for Lemma \ref{lem:intandintgivesGE}, we retain a small portion of this term when applying Young's inequality: 
\[
  -\io \frac{|\na u|^2}u + \left(\chi+2a +\frac b 2\right) \io\frac{\na u \cdot \na v} v \leq \big(-1 + (1 -\delta)\big) \io \frac{|\na u|^2}u + \frac{(\chi+2a+\frac b2)^2}{4 (1-\delta)} \io \frac{u|\na v|^2}{v^2},
\]
so that 
\begin{align*}
 \ddt \Fab(u,v)(t) 
\leq & -\delta \io \frac{|\na u|^2} u + \left(\frac{(\chi+2a+\frac b2)^2}{4 (1-\delta)}-a\chi-a-\frac b4\right) \io \frac{u |\na v|^2}{v^2}\\&+ a\io u- a\io \frac{u^2}v
 -\frac b 2 c_0 \io \frac{|\na v|^4}{v^3} - \frac b4 \io \frac{|\na v|^2} v.
\end{align*}
By virtue of the presence of $-\io \frac{|\na v|^4}{v^3}$, which originates from the additional summand of the energy functional and the preparations of Section \ref{sec:estimates},  
we can continue estimating $\io \frac{u |\na v|^2}{v^2}$ by $\io \frac{u^2}v$ and $\io \frac{|\na v|^4}{v^3}$ 
and still hope for negative coefficients in front of the integrals, in contrast to the situation of Lemma \ref{lem:ddtFa0}.
In doing so we keep some part of $\io \frac{u^2}v$ for the sake of a later application of Lemma \ref{lem43} and arrive at 
\begin{align*}
 \ddt \Fab(u,v)(t) 
\leq & -\delta \io \frac{|\na u|^2} u + a (1-\delta) \io \frac{u^2}v + \frac1{4a(1-\delta)}\left(\frac{(\chi+2a+\frac b2)^2}{4 (1-\delta)}-a\chi-a-\frac b4\right)_+^2 \io \frac{|\na v|^4}{v^3}\\ &+ a \io u - a\io \frac{u^2}v
 -\frac b 2 c_0 \io \frac{|\na v|^4}{v^3} - \frac b4 \io \frac{|\na v|^2} v,
\end{align*}
which amounts to \eqref{eq:lem47}.
\end{proof}

\begin{lemma}\label{lem44}
Let $a>0$, $b\geq 0$, $\chi>0$ be such that  
\begin{equation}\label{eq:defphii}
 \phii(a,b;\chi) := \left( \frac1{4a}\left(\frac{(\chi+2a+\frac b2)^2}{4}-a\chi-a -\frac b4\right)_+^2 - \frac{b c_0}2 \right)<0,
\end{equation}
and let $(u,v)$ be a solution of \eqref{eq:CT}.
Then there are $\kappa, \delta >0$ and $c>0$ such that for any $t\in(0,\Tmax)$, 
\[
 \ddt \Fab(u,v)(t) + \kappa \Fab(u,v)(t) +\delta \io \frac{|\na u(t)|^2}{u(t)} \leq c.
\]
\end{lemma}
\begin{proof}
By continuity of 
\[
 \delta\mapsto \phii_\delta(a,b;\chi):=\left( \frac1{4a(1-\delta)}\left(\frac{(\chi+2a+\frac b2)^2}{4 (1-\delta)}-a\chi-a -\frac b4\right)_+^2 - \frac{b c_0}2 \right)
\]
in $\delta=0$, for fixed $a,b,\chi$, negativity of $\phii(a,b;\chi)$ entails the existence of $\delta>0$ so that $\phii_\delta(a,b;\chi)$ is negative as well.
Therefore, by Lemma \ref{lem42}, 
\[
 \ddt \Fab(u,v) +\delta \io \frac{|\na u|^2}u + \delta \io \frac{u^2}v + b \io |\na\sqrt{v}|^2 \leq a\io u \qquad \mbox{on }(0,\Tmax).
\]
Since $\io u$ is constant in time by \eqref{eq:intuconst}, Lemma \ref{lem43} implies the assertion.
\end{proof}

\begin{lemma}\label{lem45}
 If 
\[
 \chi_0\in\bigset{\chi>0; \mbox{there are }a>0\mbox{ and }b\geq 0\mbox{ such that } \phii(a,b;\chi)<0} =: M,
\]
then 
\[
 (0,\chi_0)\subset M.
\]
\end{lemma}
\begin{proof}
 Since, for any fixed $a>0$, $b\geq 0$, 
\[
 \chi\mapsto \phii(a,b;\chi) = \frac1{64a}\left(\left(\chi^2+4a^2+\frac{b^2}4+b\chi+2ab-4a-b\right)_+^2-32abc_0\right)
\]
is monotone
, for any $a>0$, $b\geq 0$
\[
\phii(a,b;\chi_0)<0 \mbox{ implies } \phii(a,b;\chi)<0 \mbox{ for any } 0<\chi<\chi_0.\qedhere
\]
\end{proof}

\begin{lemma}\label{lem46}
There is $\chi_0>1$ such that $\phii(a,b;\chi_0)<0$ for some $a>0, b>0$.
\end{lemma}
\begin{proof}
 Since $\phii(\frac12,0,1)=0$ and 
\[
\frac{d}{db} \phii\left(\frac12,b,1\right)\Bigg\rvert_{b=0}= \frac{d}{db} \left(\frac1{32}\left(\frac{b^2}4+b\right)^2-\frac12 c_0b\right)\Bigg\rvert_{b=0} =  \left[\frac1{16}(\frac{b^2}4+b)(\frac b2 +1)-\frac12 c_0\right]\Bigg\rvert_{b=0}=-\frac{c_0}2<0, 
\]
 there is $b>0$ such that $\phii(\frac12,b,1)<0$ and by continuity of $\phii$ with respect to $\chi$, the assertion follows.
\end{proof}

\begin{proof}[Proof of Theorem \ref{thm:main}]
 By Lemma \ref{lem46}, there are $a,b>0$, $\chi_0>1$ such that $\phii(a,b,\chi_0)<0$ and hence, by Lemma \ref{lem45}, also $\phii(a,b,\chi)<0$ for $\chi\in(0,\chi_0)$. An application of Lemma \ref{lem44} thus reveals that for all $t>0$
\begin{equation}
\label{eq:ddtFab}
 \ddt \Fab(u,v)(t) + \kappa \Fab(u,v)(t) +\delta \io \frac{|\na u|^2}u \leq c
\end{equation}
 for some $\kappa,\delta,c>0$. Together with the boundedness of $\Fab(u,v)$ from below by Lemma \ref{lem:bdbelow} i) this ensures that $\Fab(u,v)$ is bounded so that an integration of \eqref{eq:ddtFab} also shows the boundedness of $\int_t^{t+1} \io \frac{|\na u|^2}u$.\\
Since $\Fab(u,v)$ is bounded, by Lemma \ref{lem:bdbelow} ii) the same holds true for $\io u\ln u$ and so the conditions of Lemma \ref{lem:intandintgivesGE} are met and Theorem \ref{thm:main} follows.
\end{proof}

\begin{remark}
 Assuming $c_0=\frac1{(2+\sqrt{2})^2}$, as permitted by Remark \ref{rem:c0}, 
 \[
  -1.1 \cdot 10^{-5}\approx \phii(0.49,0.001;1.015)<0,
 \]
 i.e. $\chi_0> 1.015$.
\end{remark}

%
%
%
%

%

\small

\begin{thebibliography}{10}

\bibitem{BBTW}
N.~Bellomo, A.~Bellouquid, Y.~Tao, and M.~Winkler.
\newblock Toward a mathematical theory of {K}eller-{S}egel models of pattern
  formation in biological tissues.
\newblock preprint.

\bibitem{Biler99}
P.~Biler.
\newblock Global solutions to some parabolic-elliptic systems of chemotaxis.
\newblock {\em Adv. Math. Sci. Appl.}, 9(1):347--359, 1999.

\bibitem{dPGG}
R.~{Dal Passo}, H.~Garcke, and G.~Gr\"un.
\newblock On a fourth-order degenerate parabolic equation: Global entropy
  estimates, existence, and qualitative behavior of solutions.
\newblock {\em SIAM J. Math. Anal.}, 29(2):321--342, 1998.

\bibitem{Fujie}
K.~Fujie.
\newblock Boundedness in a fully parabolic chemotaxis system with singular
  sensitivity.
\newblock {\em J. Math. Anal. Appl.}, 424(1):675 -- 684, 2015.

\bibitem{fujie_winkler_yokota_logsource}
K.~Fujie, M.~Winkler, and T.~Yokota.
\newblock Blow-up prevention by logistic sources in a parabolic-elliptic
  {K}eller-{S}egel system with singular sensitivity.
\newblock {\em Nonlinear Anal. Theory, Meth. Appl.}, 109(0):56 -- 71, 2014.

\bibitem{fujie_winkler_yokota}
K.~Fujie, M.~Winkler, and T.~Yokota.
\newblock Boundedness of solutions to parabolic-elliptic {K}eller-{S}egel
  systems with signal-dependent sensitivity.
\newblock {\em Math. Methods Appl. Sci.}, pages n/a--n/a, 2014.

\bibitem{HillenPainter}
T.~Hillen and K.~J. Painter.
\newblock A user's guide to {PDE} models for chemotaxis.
\newblock {\em J. Math. Biol.}, 58(1-2):183--217, 2009.

\bibitem{HPW}
T.~Hillen, K.~J. Painter, and M.~Winkler.
\newblock Convergence of a cancer invasion model to a logistic chemotaxis
  model.
\newblock {\em Math. Models Methods Appl. Sci.}, 23(1):165--198, 2013.

\bibitem{Horstmann03}
D.~Horstmann.
\newblock From 1970 until present: the {K}eller-{S}egel model in chemotaxis and
  its consequences. {I}.
\newblock {\em Jahresber. Deutsch. Math.-Verein.}, 105(3):103--165, 2003.

\bibitem{Horstmann04}
D.~Horstmann.
\newblock From 1970 until present: the {K}eller-{S}egel model in chemotaxis and
  its consequences. {II}.
\newblock {\em Jahresber. Deutsch. Math.-Verein.}, 106(2):51--69, 2004.

\bibitem{KS}
E.~F. Keller and L.~A. Segel.
\newblock Initiation of slime mold aggregation viewed as an instability.
\newblock {\em J. Theor. Biol.}, 26(3):399--415, 1970.

\bibitem{KStrav}
E.~F. Keller and L.~A. Segel.
\newblock Traveling bands of chemotactic bacteria: A theoretical analysis.
\newblock {\em J. Theor. Biol.}, 30(2):235 -- 248, 1971.

\bibitem{Lions}
P.~Lions.
\newblock R\'esolution de probl\`emes elliptiques quasilin\'eaires.
\newblock {\em Arch. Rational Mech. Anal.}, 74(4):335--353, 1980.

\bibitem{MPS}
R.~Man{\'a}sevich, Q.~H. Phan, and {\relax Ph}.~Souplet.
\newblock Global existence of solutions for a chemotaxis-type system arising in
  crime modelling.
\newblock {\em European J. Appl. Math.}, 24(2):273--296, 2013.

\bibitem{NSY1997}
T.~Nagai, T.~Senba, and K.~Yoshida.
\newblock Global existence of solutions to the parabolic systems of chemotaxis.
\newblock {\em S\=urikaisekikenky\=usho K\=oky\=uroku}, (1009):22--28, 1997.

\bibitem{Nirenberg_59}
L.~{Nirenberg}.
\newblock {On elliptic partial differential equations.}
\newblock {\em {Ann. Sc. Norm. Super. Pisa, Sci. Fis. Mat., III. Ser.}},
  13:115--162, 1959.

\bibitem{WkCSt}
C.~Stinner and M.~Winkler.
\newblock Global weak solutions in a chemotaxis system with large singular
  sensitivity.
\newblock {\em Nonlinear Anal. Real World Appl.}, 12(6):3727 -- 3740, 2011.

\bibitem{TaoWk}
Y.~Tao and M.~Winkler.
\newblock Eventual smoothness and stabilization of large-data solutions in a
  three-dimensional chemotaxis system with consumption of chemoattractant.
\newblock {\em J. Differential Equations}, 252(3):2520--2543, 2012.

\bibitem{wk_aggrvsglobdiffbeh}
M.~Winkler.
\newblock Aggregation vs. global diffusive behavior in the higher-dimensional
  {K}eller-{S}egel model.
\newblock {\em J. Differential Equations}, 248(12):2889--2905, 2010.

\bibitem{Wk}
M.~Winkler.
\newblock Global solutions in a fully parabolic chemotaxis system with singular
  sensitivity.
\newblock {\em Math. Methods Appl. Sci.}, 34(2):176--190, 2011.

\bibitem{wkfluid}
M.~Winkler.
\newblock Global large-data solutions in a chemotaxis-({N}avier-){S}tokes
  system modeling cellular swimming in fluid drops.
\newblock {\em Comm. Partial Differential Equations}, 37(2):319--351, 2012.

\end{thebibliography}

\end{document}